\documentclass[12pt,table]{article}
\usepackage{amsmath,amsthm,amssymb}
\usepackage{enumitem}
\usepackage[font=small,labelfont=bf]{caption}

\theoremstyle{plain}
\newtheorem{theorem}{Theorem}

\newtheorem{lemma}[theorem]{Lemma}
\newtheorem{corollary}[theorem]{Corollary}
\newtheorem{conjecture}[theorem]{Conjecture}

\theoremstyle{definition}
\newtheorem{remark}[theorem]{Remark}

\usepackage{wasysym}
\usepackage[usenames,dvipsnames]{color}
\usepackage{tikz}
\usetikzlibrary{decorations.pathmorphing}
\usepackage{xcolor}
\usepackage{pgfplots}
\usetikzlibrary{decorations.markings}
\usetikzlibrary{arrows}
\usetikzlibrary{arrows.meta}
\usetikzlibrary{shapes.geometric}
\usetikzlibrary{shapes,backgrounds}
\usepgfplotslibrary{fillbetween}
\usetikzlibrary{patterns}
\tikzstyle{empty}=[circle,draw=black!80,thick]
\tikzstyle{emptyn}=[circle,draw=black!80,fill=white,scale=0.5] 
\tikzstyle{nero}=[circle,draw=black!80,fill=black!80,thick] 
\usetikzlibrary{positioning, fit, calc}

\setlist[enumerate]{align=left}

\newcommand{\sm}{\setminus}

\newcommand{\eps}{\varepsilon}

\DeclareMathOperator{\lin}{lin}

\begin{document}
	\title{The Brown--Erd\H{o}s--S\'os Conjecture for hypergraphs of large uniformity}
	
	\author{Peter Keevash\thanks{Mathematical Institute, University of Oxford, Andrew Wiles Building, Radcliffe Observatory Quarter, Woodstock Road, Oxford, United Kingdom. E-mail: \texttt{\{keevash, longj\}@maths.ox.ac.uk}. \newline \hspace*{1.8em}Research supported
in part by ERC Consolidator Grant 647678.
} \and Jason Long\footnotemark[1]}
	
	\maketitle
	
\begin{abstract}
We prove the well-known Brown--Erd\H{o}s--S\'os Conjecture for hypergraphs of large uniformity in the following form:\ any dense linear $r$-graph $G$ has $k$ edges spanning at most $(r-2)k+3$ vertices, provided the uniformity $r$ of $G$ is large enough given the linear density of $G$, and the number of vertices of $G$ is large enough given $r$ and $k$.
\end{abstract}
	
	\section{Introduction}

	In an $r$-graph (i.e.~an $r$-uniform hypergraph), an \emph{$(s,t)$-configuration} is a collection of $t$ edges which span at most $s$ vertices. This paper concerns the following conjecture of Brown, Erd\H{o}s and S\'os~\cite{BES}, which has become one of the most well-known open problems in Extremal Combinatorics, repeatedly revisited by Erd\H{o}s in his problem papers (e.g.~\cite{Erdos81}).

	\begin{conjecture}[Brown--Erd\H{o}s--S\'os]\label{BESconj}
	For any $r>t \ge 2$ and $k\ge 3$ any $r$-graph on $n$ vertices
	with no $((r-t)k+t+1,k)$-configuration has $o(n^{t})$ edges.
	\end{conjecture}
	
	Despite receiving much attention, this conjecture remains open in almost all cases. It is well-known (see e.g.~\cite[Proposition 1.2]{CGLS}) that the conjecture would follow from the case $t=2$, in which case the conjecture reduces (see e.g.~\cite{S74}) to the following statement on hypergraphs that are \emph{linear}, meaning that no two distinct edges intersect in more than one vertex. Given a linear $r$-graph $G$ on $n$ vertices, we define its \emph{linear density} by $d_{\lin}(G) = e(G) \tbinom{r}{2} / \tbinom{n}{2}$.
		
	\begin{conjecture}\label{BESconj2}
		For any $\eps>0$ and $r,k\ge 3$ there exists $n_0=n_0(\eps, r, k)$ such that for all $n\ge n_0$ any linear $r$-graph $G$ on $n$ vertices with no $((r-2)k+3,k)$-configuration has $d_{\lin}(G) < \eps$.
	\end{conjecture}
	
	The importance of this conjecture may be gauged from the large literature surrounding even its simplest cases. The first breakthrough was achieved by Ruzsa and Szemer\'edi~\cite{RS}, who proved their celebrated `(6,3) Theorem' (the case $k=r=3$) via Szemer\'edi's regularity lemma~\cite{SRL}, which implies Roth's theorem~\cite{Roth} on 3-term arithmetic progressions via the `Triangle Removal Lemma' (see \cite{RothFromTR} and the survey \cite{CF} for more on removal lemmas and their many applications). Erd\H{o}s, Frankl and R\"{o}dl~\cite{EFR} then established the case $k=3$ for all $r$. However, the conjecture remains open for $k>3$; even the `(7,4) Problem' for $3$-graphs is generally considered to be very challenging. The most recent progress towards the conjecture includes work of Conlon, Gishboliner, Levanzov and Shapira~\cite{CGLS} showing that one can find a configuration in which the number of additional vertices beyond the conjecture is $O(\log k/\log \log k)$, independent work of Nenadov, Sudakov and Tyomkyn~\cite{others} and Long~\cite{me} establishing the conjecture when $G$ has `group structure', and a Ramsey variant due to Shapira and Tyomkyn~\cite{ST}.
		 
	 Our main result establishes Conjecture~\ref{BESconj2} when $r$ is sufficiently large given $\eps$. 
	 
	 \begin{theorem}\label{densityTS}
	 	For any $\eps>0$ there is $r_0=r_0(\eps)$ such that for all $r\ge r_0$ and for all $k\ge 3$ there exists $n_0=n_0(r,k)$ such that any linear $r$-graph $G$ on $n \ge n_0$ vertices with no $((r-2)k+3,k)$-configuration has $d_{\lin}(G) < \eps$.
	  \end{theorem}
	
We prove Theorem \ref{densityTS} in the next section, building on the methods of Shapira and Tyomkyn~\cite{ST}, who introduced the `bow-tie graph' that plays a key role in the argument. In the final section we conclude with some remarks regarding the tightness of our result and its relationship to the rich literature around the old conjecture of Erd\H{o}s~\cite{ErdosGirth} that there exist Steiner triple systems of arbitrarily high girth.
	 
\section{Proof}

The structure of our argument is broadly similar to that of Shapira and Tyomkyn~\cite{ST}. We start in the next subsection by defining the bow-tie graph and showing that it contains either a large component or many dense components. We analyse these two cases separately in the two subsequent subsections, then complete the proof in the final subsection.

\subsection{The bow-tie graph}

Given a linear $r$-graph $G$, we define the \emph{bow-tie graph} $B(G)$ as follows. The vertex set is the collection of all unordered pairs $(e,f)$ of distinct edges in $G$ such that $|e\cap f| = 1$. The edges of $B(G)$ are given as a union of triangles: given 3 distinct edges $e,f,g\in G$ such that $|e\cap f|=|e\cap g|=|f\cap g|=1$ and $e\cap f \cap g=\emptyset$, we place a triangle between the vertices $(e,f)$, $(e,g)$ and $(f,g)$ in $B(G)$. We call the vertices of $B(G)$ \emph{bow-ties}. Figure~\ref{fig2} shows a bow-tie in $G$, and the configuration in $G$ corresponding to a triangle in $B(G)$.

\begin{figure}
	\centering
		\begin{tikzpicture}[scale=0.7, every node/.style={scale=0.7}]
	\node[xshift=-5cm, yshift=0cm, label={[anchor=south,above=-1mm]\large{$(e_1,e_2)$}}] (C) {\large{$\bullet$}};
	
	\node[red, xshift=-1.6cm, yshift=-1.6cm] {\large{$e_1$}};
	\node[blue, xshift=3.6cm, yshift=-1.6cm] {\large{$e_2$}};
	
	\draw[stealth-stealth,decorate,decoration={snake,amplitude=3pt,pre length=5pt,post length=3pt}] (-4.2,0) -- ++(2,0);
	
	\draw[red] (0,0)	circle (1.8cm);

	\draw[blue] (2,0)	circle (1.8cm);
	
	\node[xshift=-0.8cm, yshift=0cm] {$r-1$};
	\node[xshift=2.8cm, yshift=0cm] {$r-1$};
	\node[xshift=1cm, yshift=0cm] {\large{$1$}};

	\node[xshift=-5cm, yshift=-2.5cm] {};
	\end{tikzpicture}
	\hspace{1.5cm}
	\begin{tikzpicture}[scale=0.7, every node/.style={scale=0.7}]
	\node[xshift=-6cm, yshift=0.7cm, label={[anchor=south,above=-7mm]\large{$(e_1,e_2)$}}] (A) {\large{$\bullet$}};
	\node[xshift=-7cm, yshift=1.7cm, label={[anchor=south,above=-1mm]\large{$(e_1,e_3)$}}] (B) {\large{$\bullet$}};
	\node[xshift=-5cm, yshift=1.7cm, label={[anchor=south,above=-1mm]\large{$(e_2,e_3)$}}] (C) {\large{$\bullet$}};
	\draw (A.center)--(B.center)--(C.center)--(A.center);
	
	\draw[stealth-stealth,decorate,decoration={snake,amplitude=3pt,pre length=5pt,post length=3pt}] (-4.2,1.2) -- ++(2,0);
	
	
	\draw[red, rotate around={30:(0,0)}] (0,0) ellipse (30pt and 60pt);
	\draw[blue, rotate around={-30:(2,0)}] (2,0) ellipse (30pt and 60pt);
	\draw[green!50!black] (1,1.732) ellipse (60pt and 30pt);

	\node[xshift=-0.3cm, yshift=-0.2cm] {\large{$r-2$}};
	\node[xshift=2.3cm, yshift=-0.2cm] {\large{$r-2$}};
	\node[xshift=1cm, yshift=2.1cm] {\large{$r-2$}};
	\node[xshift=1cm, yshift=-0.6cm] {$1$};
	\node[xshift=0cm, yshift=1.1cm] {$1$};
	\node[xshift=2cm, yshift=1.1cm] {$1$};
	\node[red, xshift=-1.6cm, yshift=-1.6cm] {\large{$e_1$}};
	\node[blue, xshift=3.6cm, yshift=-1.6cm] {\large{$e_2$}};
	\node[color = green!50!black, xshift=1cm, yshift=3.9cm] {\large{$e_3$}};

	\end{tikzpicture}

	\caption{The correspondence between vertices (left) and edges (right) of the bow-tie graph $B(G)$ and configurations in $G$. 
	}\label{fig2}
\end{figure}
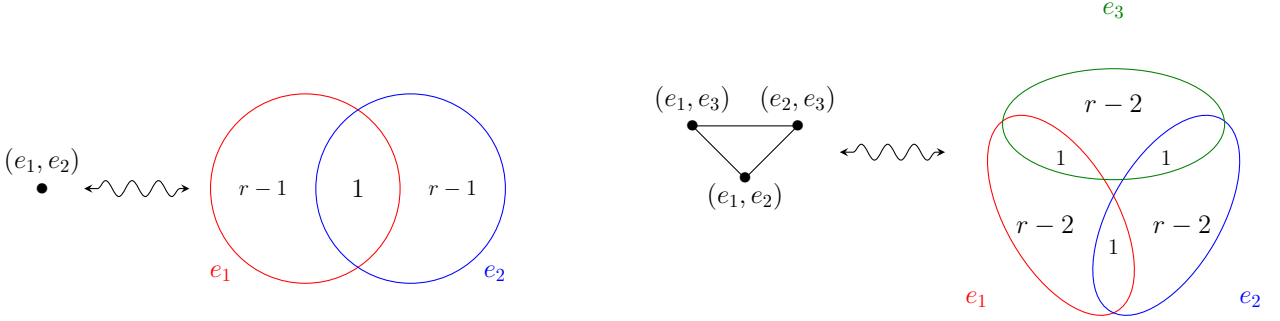

By the \emph{underlying graph} $U(G)$ we mean the graph on the same vertex set of $G$ obtained by replacing every $r$-edge of $G$ with a copy of the complete graph $K_r$ on $r$ vertices.

We begin with a lemma concerning the number of vertices of $B(G)$.

\begin{lemma}\label{vertices}
For any linear $r$-graph $G$ on $n$ vertices with
$d_{\lin}(G) \ge \eps$ and $n \ge 2r/\eps$,
the number $v(B(G))$ of vertices of $B(G)$ satisfies
	\[(\eps/2r)^2 n^3 \le v(B(G)) \le r^{-2}n^{3}.\]
\end{lemma}
\begin{proof}	
For the lower bound we note that $\sum_{v\in G}d_G(v)=re(G)
\ge r \eps \tbinom{n}{2} / \tbinom{r}{2}$, so by convexity
\[v(B(G)) = \sum_{v\in G}{d_G(v) \choose 2}
\ge n{\eps (n-1)/(r-1) \choose 2}
> (\eps/2r)^2 n^3,\]
as $n \ge 2r/\eps$. The upper bound follows from
the observation that each bow-tie contributes
exactly $(r-1)^2$ two-edge paths to the underlying graph $U(G)$;
there are clearly at most $\tbinom{n}{3}$ such paths,
and those corresponding to distinct bow-ties
are distinct by linearity of $G$.
\end{proof}

Next we prove a lower bound on the number of edges of $B(G)$ via the following well-known `Triangle Removal Lemma' of Ruzsa and Szemer\'edi~\cite{RS}.

\begin{theorem}[Triangle Removal Lemma] \label{trl}
For any $\eps>0$ there is $\delta>0$ such that
any graph on $n$ vertices with at most $\delta n^3$ triangles
may be made triangle-free by removing at most $\eps n^2$ edges.
\end{theorem}

\begin{lemma}\label{edges}
For any $\eps>0$ there is $\delta>0$ such that for any linear $r$-graph $G$ on $n$ vertices with $r \ge 3$, $n \ge r/\delta$ and $d_{\lin}(G)>\eps$ we have $e(B(G)) \ge \delta n^3$.
\end{lemma}
\begin{proof}
We bound the number of triangles in the underlying graph $U(G)$. 
Recall that the edges of $G$ correspond to edge-disjoint copies of $K_r$ in $U(G)$, so $e(U(G)) \ge \eps \tbinom{n}{2}$. 
We must remove at least one third of the edges of $U(G)$ to make it triangle-free, as this is true for each copy of $K_r$. 
Thus by Theorem \ref{trl} there is $\delta>0$ depending on $\eps$ but not on $r$ such that $U(G)$ has at least $\delta n^3$ triangles.
The number of such triangles contained within some edge of $G$ is
$\tbinom{r}{3} e(G) < rn^2/6 < \delta n^3/2$.
The other triangles determine edge-disjoint triangles in $B(G)$. 
\end{proof}

We say that a component $\mathcal{C}$ of the bow-tie graph $B(G)$ is \emph{dense} if it has average degree greater than $12r$. 
We conclude this subsection by presenting a lemma showing that $B(G)$ contains a large component or many dense components.

\begin{lemma}\label{dichotomy}
For any $\eps>0$ there is $r_0=r_0(\eps)$ such that for any $r \ge r_0$ and $k \ge 2$ there is $n_0=n_0(r,k)$ such that for any linear $r$-graph $G$ on $n \ge n_0$ vertices with $d_{\lin}(G) \ge \eps$,  
in $B(G)$ we have a component with at least $k^2$ vertices 
or at least $(\eps/2rk^2)^2 n^3$ dense components.
\end{lemma}
\begin{proof}
We may assume $n$ is large enough to apply Lemmas~\ref{vertices} and~\ref{edges}, so $B(G)$ has at least $(\eps/2r)^2 n^3$ vertices and average degree at least $\delta(\eps)r^2$. We may also assume that $B(G)$ has no component with at least $k^2$ vertices. To prove the lemma, it now suffices to show that a proportion at least $k^{-2}$ of the vertices of $B(G)$ belong to dense components. If this were false then, as the average degree in dense components is at most $k^2$ and in non-dense components at most $12r$, the average degree in $B(G)$ would be at most $12r + k^{-2} k^2 = 13r$. However, for $r>r_0(\eps)$ large we have $13r<\delta(\eps)r^2$, which is a contradiction. 
\end{proof}

\subsection{Large component}

Given a subset $\mathcal{C}$ of the vertices of $B(G)$, we write $G(\mathcal{C})$ for the subgraph of $G$ obtained by taking all edges belonging to a bow-tie in $\mathcal{C}$. In this subsection we establish the following lemma showing that we can find the required configuration in $G$ whenever $B(G)$ has a component $\mathcal{C}$ such that $G(\mathcal{C})$ has many edges.

\begin{lemma}\label{largeCtool}
	Let $G$ be a linear $r$-graph. Suppose $B(G)$ contains a component $\mathcal{C}$ such that $G(\mathcal{C})$ contains at least $k$ edges. Then $G$ contains an $((r-2)k+3,k)$-configuration.
\end{lemma}

\begin{proof}
We may assume $k\ge 3$, since $G(\mathcal{C})$ is a union of triangles and the case $k<3$ is trivial.

We grow the required configuration in $G$ inductively as follows. We start by fixing any triangle ${\cal C}_1$ in ${\cal C}$. We note that $G(\mathcal{C}_1)$ is a $(3(r-2)+3,3)$-configuration in $G$.

Given an induced proper subgraph ${\cal C}_i$ of ${\cal C}$ with $i \ge 1$, to obtain $\mathcal{C}_{i+1}$ we begin by selecting a vertex $b_i\in \mathcal{C}\setminus \mathcal{C}_i$ which belongs to the neighbourhood of a vertex in $\mathcal{C}_i$.  Suppose that $b_i$ represents the bow-tie $(e_i,f_i)$. Since $b_i$ belongs to the neighbourhood of $\mathcal{C}_i$, one of $e_i$ or $f_i$, say $e_i$, must belong to $G(\mathcal{C}_i)$. To obtain $\mathcal{C}_{i+1}$ we add to $\mathcal{C}_i$ all vertices of ${\cal C}$ corresponding to a bow-tie containing $f_i$ and another edge $e\in G(\mathcal{C}_i)$; in particular $b_i$ is such a vertex.

Note that at each step we introduce precisely one new edge $f_i$ into $G(\mathcal{C}_i)$, so the number $E_i$ of edges in $G(\mathcal{C}_i)$ is equal to $i+2$. Moreover, we introduce at most $r-2$ new vertices, as $f_i$ contains $2$ vertices in $e_i\cup e\subset G(\mathcal{C}_i)$ by choice of $b_i$ (see Figure~\ref{fig2}). Thus the number of vertices of $G$ spanned by $G(\mathcal{C}_i)$ is at most $(r-2)(i+2)+3=(r-2)E_i+3$. As $e(G(\mathcal{C}))\ge k$ we reach the $((r-2)k+3,k)$-configuration $G(\mathcal{C}_{k-2})$.
\end{proof}

It will be helpful to deduce the lower bound on the number of edges of $G(\mathcal{C})$ required by Lemma~\ref{largeCtool} from a lower bound on the number of vertices of $\mathcal{C}$. Indeed, as vertices of $B(G)$ correspond to pairs of edges of $G(\mathcal{C})$, if $B(G)$ contains a component $\mathcal{C}$ with at least $k^2$ vertices then $G(\mathcal{C})$ contains at least $k$ edges, so we may apply Lemma~\ref{largeCtool}.

\begin{corollary}\label{largeC}
	Let $G$ be a linear $r$-graph such that $B(G)$ contains a component $\mathcal{C}$ with at least $k^2$ vertices. Then $G$ contains an $((r-2)k+3,k)$-configuration.
\end{corollary}

\subsection{Dense components}

In this subsection we prove the following lemma showing that dense components provide configurations using so few vertices that we can combine several of them in proving Theorem~\ref{densityTS}.
	
\begin{lemma}\label{denseCs}
Let $\mathcal{C}$ be a dense component of $B(G)$. Then $G(\mathcal{C})$ is an $((r-2)u,u)$-configuration, where $u=e(G(\mathcal{C}))$.
\end{lemma}
\begin{proof}
Let $({\cal C}_i)_{i \ge 1}$ be as in the proof of Lemma~\ref{largeCtool}. Thus $G({\cal C}_1)$ is a $(3(r-2)+3,3)$-configuration in $G$, at each step we add one new edge of $G$ and at most $r-2$ new vertices of $G$, so $G({\cal C})$ spans at most $u(r-2)+3$ vertices of $G$. To improve on this upper bound by $3$, and so prove the lemma, it suffices to show that there is some step $i$ where the new edge $f_i$ shares at least $5$ vertices with $G({\cal C}_i)$, and so we add at most $(r-2)-3$ new vertices at this step.

Suppose for a contradiction that at each step $i$ the new edge $f_i$ of $G$ shares fewer than $5$ vertices with $G(\mathcal{C}_i)$. For each $i$ we consider $\mathcal{C}_{i+1}\sm\mathcal{C}_i$, denoting its vertex set by $V_i$ and its edge set by $E_i$. We will show that $|E_i| \le 6r|V_i|$. As ${\cal C}_1$ has $3$ vertices and $3$ edges this will imply that each ${\cal C}_i$ has average degree at most $12r$, giving the required contradiction, as $\mathcal{C}$ is dense.

	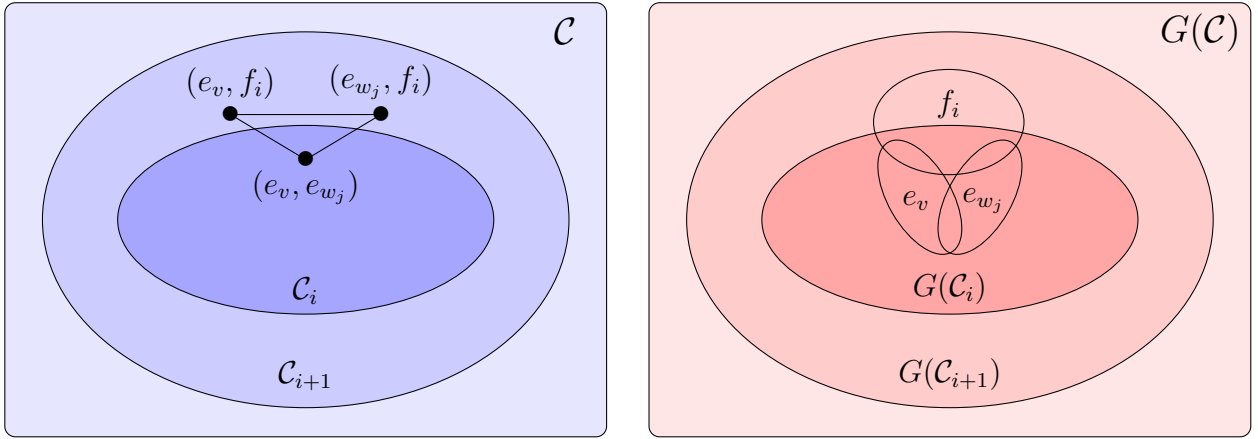
\begin{figure}
		\centering
		\begin{tikzpicture}
		\node[rounded corners, draw, fill=blue!10 ,text height = 5.5cm,minimum     width=8cm, label={[anchor=north west,above=-7mm]40:\large{$\mathcal{C}$}}] (main){};
		\node[ellipse, draw, fill=blue!20, text height =3cm, minimum width = 7cm, minimum height=5cm, label={[anchor=south,above=1mm]270:$\mathcal{C}_{i+1}$}] at (main.center)  (semi) {};
		\node[ellipse, draw,fill=blue!35, text height = 1.5cm, minimum width = 5cm,label={[anchor=south,above=0mm]270:$\mathcal{C}_i$}] at (main.center) (active) {};
		\node[xshift=-1cm, yshift=1.4cm, label={[anchor=south,above=-2mm]$(e_v,f_i)$}] at (main.center) (A) {\large{$\bullet$}};
		\node[xshift=1cm, yshift=1.4cm, label={[anchor=south,above=-2mm]$(e_{w_j},f_i)$}] at (main.center) (B) {\large{$\bullet$}};
		\node[xshift=0cm, yshift=0.8cm, label={[anchor=south,above=-10mm]$(e_{v},e_{w_j})$}] at (main.center) (C) {\large{$\bullet$}};
		\draw (A.center) -- (B.center) -- (C.center) -- (A.center);
		
		\end{tikzpicture}
		\quad
		\begin{tikzpicture}
		\node[rounded corners, draw, fill=red!10 ,text height = 5.5cm,minimum     width=8cm, label={[anchor=north west,above=-8mm]41:\large{$G(\mathcal{C})$}}] (main){};
		\node[ellipse, draw, fill=red!20, text height =3cm, minimum width = 7cm, minimum height=5cm, label={[anchor=south,above=1mm]270:$G(\mathcal{C}_{i+1})$}] at (main.center)  (semi) {};
		\node[ellipse, draw,fill=red!35, text height = 1.5cm, minimum width = 5cm,label={[anchor=south,above=0mm]270:$G(\mathcal{C}_i)$}] at (main.center) (active) {};
		\node[ellipse, draw,  align=left, anchor=east, xshift=-1.515cm, yshift=1.3cm, minimum height=1.4cm, minimum width=2cm , label={[anchor=south,above=-8mm]90:$f_i$} ] at (active.east) (fi) {};
		\draw[rotate around={30:(-0.4,0.3)}] (-0.4,0.3) ellipse (12pt and 24pt);
		\draw[rotate around={-30:(0.4,0.3)}] (0.4,0.3) ellipse (12pt and 24pt);
		\node[xshift=.45cm,yshift=0.25cm] (0,0) {$e_{w_j}$};
		\node[xshift=-.45cm,yshift=0.25cm] (0,0) {$e_{v}$};
		\end{tikzpicture}
		\caption{At stage $i$, the bow-tie $v=(e_v,f_i)\in \mathcal{C}_{i+1}$ is considered. Edges between $v$ and $\mathcal{C}_i$ in $B(G)$ correspond to choices of $e_{w_j}$ in $G(\mathcal{C}_i)$ such that $(e_v,f_i,e_{w_j})$ forms a triangle in $B(G)$, as depicted in $B(G)$ on the left and in $G$ on the right. Each of $f_i\cap e_v$, $f_i\cap e_{w_j}$ and $e_v\cap e_{w_j}$ consists of a single vertex of $G(\mathcal{C}_i)$, and these vertices are distinct.}\label{fig1}
	\end{figure}

Consider any vertex $v$ of $V_i$, corresponding to some bow-tie $(e_v, f_i)$ of edges in $G$, where $e_v\in G(\mathcal{C}_i)$. We list the edges from $v$ to $V({\cal C}_i)$ as $vw_1,\dots,vw_s$. Each $w_j$ corresponds to a bow-tie $(e_v,e_{w_j})$ for some edge $e_{w_j}$ of $G(\mathcal{C}_i)$ with $|e_{w_j}\cap f_i|=1$; see Figure~\ref{fig1}. To bound $s$, we note that any edge of $G$ is determined by any two of its vertices (as $G$ is linear), so each $e_{w_j}$ is determined by $e_{w_j} \cap e_v$ and $e_{w_j}\cap f_i$. There are trivially at most $r$ choices for the former, but at most $|f_i \cap V(G({\cal C}_i))| \le 4$ choices for the latter, so the degree of $v$ satisfies $s \le 4r$.

Next we estimate the degree of $v$ within $V_i$. Consider any edge $vv_2$ of $\mathcal{C}$ with $v_2 \in V_i$. Then $v_2$ corresponds to the bow-tie $(e_{v_2},f_i)$ with $e_{v_2}\in \mathcal{C}_i$, and moreover $(e_v,e_{v_2})$ must be a bow-tie, corresponding to a vertex $w$ of $\mathcal{C}_i$ such that $vw$ is also an edge of $\mathcal{C}$. As $v$ and $w$ uniquely determine $v_2$, the degree of $v$ within $V_i$ is also at most $4r$. We deduce that ${\cal C}$ has at most $4r|V_i|$ edges between $V_i$ and $V({\cal C}_i)$ and at most $2r|V_i|$ edges within $V_i$, so $|E_i| \le 6r|V_i|$, as required.
\end{proof}

\begin{remark}
We do not attempt to optimise the density of configurations that can be obtained from the proof of Lemma~\ref{denseCs}, as this is only relevant to one case of the proof of Theorem \ref{densityTS}, and so would not give any improvement to our main result.
\end{remark}

\subsection{Completing the proof}

\begin{proof}[Proof of Theorem~\ref{densityTS}]
	We begin by applying Lemma~\ref{dichotomy}. If $B(G)$ contains a component with at least $k^2$ vertices then we are done by Corollary~\ref{largeC}. Otherwise, $B(G)$ contains $\Omega(n^{3}/r^2)$ dense components. Pick a dense component $\mathcal{C}_1$. Let $u_1$ be the number of edges of $G(\mathcal{C}_1)$. We may assume $u_1<k$, otherwise we are done by Lemma~\ref{largeCtool}.
	
By Lemma~\ref{denseCs} there is an $((r-2)u_1,u_1)$-configuration $T_1$ in $G(\mathcal{C}_1)$. We let $B_1$ be the subgraph of $B(G)$ obtained by eliminating all bow-ties containing any edge in $T_1$. We thus delete at most $k e(G) = O(n^2/r^2)$ vertices of $B(G)$, so $B_1$ still has $\Omega(n^{3}/r^2)$ dense components.

Repeating this process, we find a sequence $(T_i)_{i \ge 1}$ of edge-disjoint $((r-2)u_i,u_i)$-configurations, stopping when we first reach some $m$ with $u_1+\dots+u_m\ge k$. Write $u=u_1+\dots+u_{m-1}$. As $u_m \ge k-u$, by  Lemma~\ref{largeCtool} we can find an $((r-2)(k-u)+3,k-u)$-configuration $T_m'$ in $\mathcal{C}_m$. Then $T_1\cup\dots\cup T_{m-1}\cup T_m'$ is an $((r-2)k+3,k)$-configuration.
\end{proof}

\section{Concluding remarks}

As mentioned in the introduction, there is a straightforward reduction of the Brown-Erd\H{o}s-S\'os Conjecture to the case $t=2$. To describe this in our setting, we call a hypergraph \emph{$t$-linear} if no two edges share at least $t$ vertices. We define the $t$-linear density of a $t$-linear $r$-graph $G$ on $n$ vertices by $d_t(G) = e(G) \tbinom{r}{t} / \tbinom{n}{t}$. For any such $G$, by averaging there is a vertex $v$ with degree at least $e(G)r/n$, whose link $G(v)$ is therefore a $(t-1)$-linear $(r-1)$-graph on $n-1$ vertices with $d_{t-1}(G(v)) \ge d_t(G)$. By induction and Theorem \ref{densityTS} we deduce the following.

	 \begin{theorem}\label{densityTS2}
	 	For any $\eps>0$ and $t \ge 2$ there is $r_0=r_0(\eps,t)$ such that for all $r\ge r_0$ and for all $k\ge 3$ there exists $n_0=n_0(r,k)$ such that any $t$-linear $r$-graph $G$ on $n \ge n_0$ vertices with no $((r-2)k+t+1,k)$-configuration has $d_t(G) < \eps$.
	  \end{theorem}

By analogy with Conjecture~\ref{BESconj2}, we conjecture that the conclusion of Theorem \ref{densityTS2} holds even without the assumption that $r$ is sufficiently large. It would also be interesting to determine whether the bound in Theorem \ref{densityTS2} is tight. A simple random construction (see \cite{BES}) provides $r$-graphs on $n$ vertices with no $((r-2)k+t+1,k)$-configuration and $cn^t$ edges, but here $c$ depends on $r$. Are there $t$-linear $r$-graphs $G$ on $n$ vertices with no $((r-2)k+t+1,k)$-configuration and $d_t(G)$ bounded away from $0$ as $n \to \infty$?
	 	
An old conjecture of Erd\H{o}s~\cite{ErdosGirth}, related to the Brown-Erd\H{o}s-S\'os Conjecture, suggests that there exist Steiner triple systems of arbitrarily high girth. In the terminology of this paper, this can be restated as follows: for any $g \ge 2$ there is a complete linear $3$-graph on $n$ vertices (i.e.\ with every pair covered by an edge) with no $(g+2,g)$-configuration. It is easy to show (see e.g.~\cite[Proposition 7.1]{GKLO}) that complete linear $3$-graphs have $(g+3,g)$-configurations, so this conjecture of Erd\H{o}s would imply tightness of the Brown-Erd\H{o}s-S\'os Conjecture in a strong sense. Furthermore, the weaker question (posed by Lefmann, Phelps and R\"odl~\cite{LPR} and by Ellis and Linial~\cite{EL}) of whether there are linear $3$-graphs with large girth and linear density bounded away from zero was also open until its recent solution independently by Bohman and Warnke~\cite{BW} and by Glock, K\"uhn, Lo and Osthus~\cite{GKLO} in a strong form:~they showed that there are linear $3$-graphs with large girth and linear density approaching $1$.
 
Glock, K\"uhn, Lo and Osthus (see \cite[Conjecture 7.2]{GKLO}) also pose an extension of the conjecture of Erd\H{o}s to Steiner $(n,r,t)$-systems of large girth. In our terminology, for any $k,r,t$ we seek complete $t$-linear $r$-graphs with no $((r-t)k+t+1,k+1)$-configuration, on any set of $n$ vertices where $n$ is large and \emph{admissible}, in that there exist complete $t$-linear $r$-graphs on $n$ vertices (such $n$ were characterised by Keevash~\cite{K}). They show (see \cite[Theorem 7.5]{GKLO}) that there are $t$-linear $r$-graphs with no $((r-t)k+t+1,k+2)$-configuration and $t$-linear density approaching $1$, thus relaxing the conjecture in two ways (allowing an extra edge and relaxing `complete' to `almost complete'). We conjecture the following stronger extension focussing on the minimum number of vertices spanned by a given number of edges:~for any $k,r,t$ there should exist complete $t$-linear $r$-graphs with no $((r-t)k+t,k)$-configuration, and furthermore these should exist on $n$ vertices whenever $n$ is large and admissible.

\end{document}